\documentclass[11pt]{amsart}

\usepackage{graphicx}
\usepackage{tikz}
\usepackage{amsthm}
\usepackage{amssymb}
\usepackage{todonotes}
\usepackage[colorlinks=true]{hyperref}
\usepackage[bibtex-style]{amsrefs}
\usepackage{enumitem}
\theoremstyle{plain}
\newtheorem{theorem}{Theorem}
\newtheorem{conjecture}[theorem]{Conjecture}
\newtheorem{proposition}[theorem]{Proposition}
\newtheorem{lemma}[theorem]{Lemma}
\newtheorem{corollary}[theorem]{Corollary}

\theoremstyle{definition}
\newtheorem{definition}[theorem]{Definition}
\newtheorem{claim}{Claim}[theorem]

\title{Contracting a Single Element in a Transversal Matroid}

\author{Sam Bastida}

\DeclareMathOperator{\down}{down}

\begin{document}

\maketitle

\begin{abstract}
    It is well known that the class of transversal matroids is not closed under contraction or duality. In particular, after contracting a set of elements from a transversal matroid, the resulting matroid may or may not be transversal, and the computational complexity of determining whether it is transversal is not well understood. As a step toward resolving this, we provide a polynomial-time algorithm for determining whether a single-element contraction of a transversal matroid is transversal. In the case that the single-element contraction is transversal, the algorithm also provides a transversal representation. We then discuss possible applications of our techniques toward finding a polynomial-time algorithm to determine if the dual of a transversal matroid is transversal.
\end{abstract}

\section{Introduction}

\subsection{Motivation}

Transversal matroids are an important class of matroids. In introductory texts to matroid theory the theory of transversal matroids is often well developed \cite{welsh}, and they are regularly given as an example of naturally arising matroids alongside the original classes, representable and graphic matroids \cite{oxley}. However, there are some key differences that separate transversal matroids from these other two fundamental classes. Firstly, representable and graphic matroids are closed under minors while transversal matroids are not. Furthermore, obtaining representations of minors of representable or graphic matroids can be done in polynomial time while no polynomial-time algorithm is known that determines if a minor of a transversal matroid is even transversal or not, nevermind one that obtains a representation. Secondly, representable matroids are closed under duality and representations of their duals can, once again, be found in polynomial time. Graphic matroids are not closed under duality, however, the class of graphic matroids which are also co-graphic is well understood and such matroids can be recognised, and representations of their duals found, in polynomial time. Currently there is very little understanding of which transversal matroids are also co-transversal and of course there does not exist a polynomial-time algorithm to find transversal representations of duals of transversal matroids.

These ideas immediately give rise to the following natural problems. In all such problems where our problem instance contains a transversal matroid we will assume that the transversal matroid is represented by a bipartite graph in the usual way. \\
\\
\textsc{Transversal Contraction}\\
\textsc{Instance}: A transversal matroid $M$ and a subset $X \subseteq E(M)$.\\
\textsc{Question}: Is $M/X$ transversal?\\
\\
\\
\textsc{Transversal Dual}\\
\textsc{Instance}: A transversal matroid $M$.\\
\textsc{Question}: Is $M^*$ transversal?\\
\\
These are questions that have remained unresolved for a long time: Welsh asked in 1971 \cite[Problem 29]{DominicProblems} for a characterisation of the matroids that are both transversal and co-transversal; while Bonin, in 2010, asked the same thing \cite[Open Problem 6.1]{jobonotes}. We are interested in a slightly different, but related question: whether there even exists a polynomial-time algorithm for determining if a matroid is transversal and co-transversal. As we explain in Section 4 a polynomial-time algorithm that solves \textsc{Transversal Contraction} would provide a polynomial-time algorithm that solves \textsc{Transversal Dual}.

We work toward answering these questions by addressing the simplest case of \textsc{Transversal Contraction}, when $X$ consists of a single element.\\
\\
\textsc{Single Element Transversal Contraction}\\
\textsc{Instance}: A transversal matroid $M$ and an element $e \in E(M)$.\\
\textsc{Question}: Is $M/e$ transversal?\\
\\
We provide a polynomial-time algorithm that solves this problem and, if $M/e$ is determined to be transversal, our algorithm provides a bipartite graph representation. Therefore we obtain the following theorem.

\begin{theorem}
\label{theorem 1}
    There exists a polynomial time algorithm that solves \textsc{Single Element Transversal Contraction}. If this algorithm obtains a positive answer then it also provides the transversal representation of $M/e$.
\end{theorem}

We will discuss the underlying complexity theory of the problem in more detail in Section 2 but we give a brief description here to show that this problem is in $\Sigma_2$. We assume that $e$ is not a loop of $M$ as then $M/e=M\backslash e$ which is trivially transversal. We are able to check, in polynomial time, if a set $B \subseteq E(M)-e$ is a basis of $M/e$ by simply checking if $B \cup e$ is a basis of $M$. Similarly for any transversal matroid, $M_T$, with a transversal representation, we can check if $B$ is a basis of $M_T$. This means that if we are given a transversal matroid, $M_T$, on the same ground set as $M/X$ we can check in polynomial time whether a basis of one is a basis of the other. In order for us to obtain a positive answer to \textsc{Transversal Contraction} there must exist a matroid, $M_T$, such that for all $B$ we have that $B$ is not a basis of exactly one of $M/X$ or $M_T$. This alternation of quantifiers puts \textsc{Transversal Contraction} into $\Sigma_2$ \cite{npcompleteness}. Whether the problem is in fact in $NP$, or perhaps even in $P$, remains open.

\subsection{Strategy}
Our approach to obtaining a polynomial time algorithm will examine the duals of transversal matroids. The class of co-transversal matroids is exactly the class of strict gammoids \cite{INGLETON197351}. Strict gammoids have their own graphical representation \cite{mason} which we shall discuss in Section 2.

Instead of working directly with transversal representations we will work with strict gammoids. This allows us to examine restrictions of strict gammoids instead of contractions of transversal matroids. Restrictions of strict gammoids are called gammoids and they have their own graphical representation. This allows us to quickly determine any matroid properties such as independence or closure that we might need from the restricted matroid for our algorithm.

Gammoids are interesting in their own right as they are closed under minors and since any transversal matroid is also a gammoid \cite{oxley} they are a minor closed class containing transversal matroids. In fact they are exactly the class of transversal matroids and their minors \cite{oxley}. Representing gammoids graphically can be a challenge though as it was only recently that a bound was put on the size of a graphic representation of a gammoid \cite{6375323}. Our work approaches this problem from another direction, assuming that we already have a small representation of a gammoid and finding a representation that is one vertex smaller or determining that no such representation exists.

Since we are working in the dual picture we obtain a solution to the dual of \textsc{Single Element Transversal Contraction} where our strict gammoid is represented by a directed graph in the usual way.\\
\\
\textsc{Single Element Strict Gammoid Deletion}\\
\textsc{Instance}: A strict gammoid $M$ and an element $e \in E(M)$.\\
\textsc{Question}: Is $M\backslash e$ a strict gammoid?\\
\\
That is, we provide a proof for the following theorem which, through duality, also proves Theorem~\ref{theorem 1}.

\begin{theorem}
\label{theorem 2}
    There exists a polynomial time algorithm that solves \textsc{Single Element Strict Gammoid Deletion}. If this algorithm obtains a positive answer then it also provides the strict gammoid representation of $M\backslash e$.
\end{theorem}

\subsection{Further Applications}
One might wonder why a polynomial-time algorithm for a single-element contraction does not lead to a polynomial-time algorithm for a many-element contraction. As it turns out, it is possible that for a transversal matroid, $M$, there are distinct elements $e$ and $f$ such that neither $M/e$ nor $M/f$ is transversal, but $M/\{e,f\}$ is transversal. Note that, in this case, our algorithm can determine that $M/e$ and $M/f$ are not transversal, but then cannot be used to determine anything about $M/\{e,f\}$.

In general, even our most fundamental results do not apply when contracting more than one element. However, there is a specific case in which some of our results do apply. A fundamental basis of a matroid, $M$, is a basis, $B_F$, of $M$ such that for any cyclic flat $Z$ of $M$ we have that $B_F \cap Z$ spans $Z$. As we explain in Section 4, if a polynomial-time algorithm existed that could solve \textsc{Transversal Contraction} in polynomial time in the case when $X$ is a fundamental basis of $M$, then a polynomial time algorithm for \textsc{Transversal Dual} would also exist. Therefore, while we can currently only contract a single element in polynomial time, if our methods can be built upon, we may be able to provide the first real progress toward the questions in \cite{DominicProblems} and \cite{jobonotes} in over 50 years.

Also in Section 4 we explain that while many of our important results for a single element contraction do not carry over to the general case, some of them do transfer to the case when $X$ is a fundamental basis. This provides hope that perhaps our technique, or a similar technique can solve this problem and therefore provide a polynomial-time algorithm that solves \textsc{Transversal Dual}.

\subsection{Organisation}

The remainder of this paper is organised as follows. In Section 2 we introduce some basic concepts in matroid theory and the theory of transversal matroids. In Section 3 we develop the necessary theory and then prove Theorem~\ref{theorem 2} which of course also proves Theorem~\ref{theorem 1}. In Section 4 we then examine the potential of extending our ideas toward obtaining a polynomial-time algorithm to determine if a transversal matroid is co-transversal. Then in Section 5 we conclude the paper with some conjectures about extensions of our results.

\section{Preliminaries}

\subsection{Notation}
If $S$ is a set we let $S+x$ be $S \cup \{x\}$. We use $[n]$ to denote the set $\{0,1,2...,n\}$. For sets $S$ and $S'$ we use $S \subseteq S'$ to denote that $S$ is a subset of $S'$ and $S \subset S'$ to denote that $S$ is a proper subset of $S'$. If $M$ is a matroid and $X$ is a subset of $M$ we use $M|X$ to represent $M\backslash (M-X)$, that is, $M$ restricted to $X$.

\subsection{Transversals}
Let $E$ be a finite set. We define a {\em set system} $(E,\mathcal{A})$ to be $E$ along with a multiset $\mathcal{A}=(A_j: j \in [n])$ of subsets of $E$ where $n=|\mathcal{A}|$. Thus the same set may appear multiple times in $\mathcal{A}$, indexed by different elements of $[n]$.

A set system $(E,\mathcal{A})$ can be represented by a bipartite graph whose vertex set is $\mathcal{A} \cup E$; an edge connects $A_j \in \mathcal{A}$ with $e \in E$ if and only if $e \in A_j$. Similarly each bipartite graph can be seen as representing a set system.

A {\em transversal} of the set system $(E,\mathcal{A})$ is a subset $T$ of $E$ for which there is a bijection $\phi: E \rightarrow [n]$ with $e \in \phi(e)$ for all $e \in E$. We refer to the bijection, $\phi$, as a {\em matching}.

A well known theorem of Hall determines necessary and sufficient conditions for when a set system has a transversal \cite{hall1987representatives}.

\begin{theorem}
\label{hall}
    A finite set system $(E,(A_j: j \in [n]))$ has a transversal if and only if for all $K \subseteq [n]$,
    \[
    |\bigcup_{j \in K} A_j| \geq |K|.
    \]
\end{theorem}

If $T$ is a transversal of set system $(E,\mathcal{A})$ and $T'$ is a proper subset of $T$ then we say $T'$ is a {\em partial transversal} of the system $(E,\mathcal{A})$.

\subsection{Lattices}
A lattice is a set, $L$, with a partial order such that any two elements, $e_0$ and $e_1$ have a unique least upper bound, called the {\em join} of $e_0$ and $e_1$, denoted $e_0 \lor e_1$, and a unique greatest lower bound, called the {\em meet} of $e_0$ and $e_1$, denoted $e_0 \land e_1$. We will only consider finite lattices. Such a lattice will have both a minimum element, $0_L$ and a maximum element $1_L$.

We say a set system is a {\em lattice under inclusion} if the partial ordering imposed on the sets in the system by containment gives rise to a lattice.

\subsection{Matroids}

We will use the terms; {\em independent sets}, {\em dependent sets}, {\em loops}, {\em coloops}, {\em bases}, {\em circuits}, {\em cyclic sets}, {\em rank}, {\em closure}, {\em flats}, {\em cyclic flats}, {\em span}, {\em contraction}, {\em deletion}, {\em minors}, and {\em duality} without providing definitions. We refer the reader to \cite{oxley} for details on matroids.

We have from \cite{bonin2008lattice} that it is possible to define a matroid from its set of cyclic flats and their ranks in the following way:

\begin{definition}
\label{cyclic flat axioms}
    Let $(E(M),\mathcal{Z})$ be a set system where all sets in $\mathcal{Z}$ are distinct and let $r$ be an integer-valued function on $\mathcal{Z}$. There is a matroid for which $\mathcal{Z}$ is the collection of cyclic flats and $r$ is the rank function restricted to the sets in $\mathcal{Z}$ if and only if
    \begin{enumerate}
        \item [Z1] $\mathcal{Z}$ is a lattice under inclusion.
        \item [Z2] $r(0_\mathcal{Z})=0$.
        \item [Z3] $0<r(Y)-r(X)<|Y-X|$ for all sets $X,Y \in \mathcal{Z}$ with $X \subset Y$.
        \item [Z4] For all sets $X,Y \in \mathcal{Z}$,
        \[
        r(X)+r(Y) \geq r(X\lor Y) + r(X \land Y) + |(X \cap Y) - (X \land Y)|.
        \]
    \end{enumerate}
\end{definition}

We let $\mathcal{Z}(M)$ be the collection of all cyclic flats of a matroid $M$. We will also define the {\em nullity} function, $n:2^{E(M)} \rightarrow \mathbb{N}$ as $n(X)=|X|-r(X)$. Therefore we see that we can equivalently define a matroid using its set of cyclic flats and their nullities. We also have from \cite{oxley}*{Chapter 2.1 Exercise 13a} that a set is a cyclic flat of a matroid, $M$, if and only if its complement is a cyclic flat of, $M^*$, the dual of $M$. We will also make use of the following equation for the dual rank function \cite{oxley}:
\begin{equation}
\label{rank dual}
r_{M^*}(X) = r_M(\overline{X})+|X|-r_M(M).
\end{equation}

Now let $(E(M), \mathcal{A})$ be a set system and let $\mathcal{I}$ be the set of partial transversals on $(E(M), \mathcal{A})$. Then $(E(M),\mathcal{I})$ is a matroid \cite{jobonotes}. We call any matroid arising in this way a {\em transversal matroid}. We define an important function called the $\beta$ function recursively on all subsets of a matroid
\begin{equation}
\label{betaeq}
\beta(X) = r(M) - r(X) - \sum_{Z \in \mathcal{Z}(M): X \subset Z} \beta(Z).
\end{equation}

Though they did not formulate their theorem in this way we have the following theorem from \cite{ingletoninequality} and \cite{masoninequality} (the proof that this is equivalent to the formulations of \cite{ingletoninequality} and \cite{masoninequality} can be found in \cite{jobonotes}).

\begin{theorem}
\label{beta}
	 A matroid, $M$, is transversal if and only if $\beta(X)\geq 0$ for all $X \subseteq E(M)$.
\end{theorem}

Next let $D$ be a directed graph with $E$ and $S$ each being subsets of $V(D)$, we do not require that $E$ and $S$ are disjoint. We say that a set $I \subseteq E$ is {\em linked into $S$} if there exist $|I|$ disjoint directed paths each starting at a vertex in $I$ and ending at a vertex in $S$. The set system $(E,\mathcal{I})$ where $\mathcal{I}$ is the collection of sets with linkings into $S$ is a matroid \cite{oxley}. We denote this particular matroid as $M(D,E,S)$ and we call any matroid arising in this way a {\em gammoid}. If $M$ is a gammoid such that $M=M(D,V(D),S)$ for some directed graph $D$ then we say that $M$ is a {\em strict gammoid}.

It is not difficult to see that gammoids are the restrictions of strict gammoids. Strict gammoids, and as an extension, gammoids, are of interest to us due to the following theorem \cite{INGLETON197351}.

\begin{theorem}
\label{cotrans}
    A matroid is co-transversal if and only if it is a strict gammoid.
\end{theorem}

\subsection{Complexity Theory}
Let $S$ be a set of characters. We call a sequence of characters from $S$ a {\em word}. We call a set of words, $L$, a {\em language}. Complexity theory examines the {\em complexity} of algorithms that determine if a given word belongs to a given language. For our purposes we can imagine our words as inputs to a Turing Machine coded with an algorithm that will eventually accept the word if it belongs to the language and reject it otherwise. If we can bound the number of steps the Turing Machine will take by some polynomial function, $n^k$, where $n$ is the size of the word and $k$ is some constant then we say that the algorithm runs in {\em polynomial time}.

We classify languages based on their complexity and thus build a structure known as the {\em polynomial hierarchy}. The first level of the polynomial hierarchy is $P$. A language, $L$, is in $P$ if there exists an algorithm which will determine that a word belongs to $L$ in polynomial time.

The next level of the polynomial hierarchy is best viewed through an example. Let $L_S$ be the language of satisfiable boolean formulae. We let $L_S'$ be the language consisting of words of the form $(F,A)$ where $F$ is a boolean formula and $A$ is an assignment of the variables of $F$ such that $F$ evaluates to true. Clearly a word, $F$, belongs to $L_S$ if and only if there exists a word, $(F,A)$, belonging to $L_S'$. A polynomial-time algorithm for determining whether a word belongs to $L_S$ does not exist as far as we know, however, a polynomial-time algorithm for determining whether a word belongs to $L_S'$ does exist, one simply needs to evaluate $F$ with $A$ to ensure that $F$ evaluates to true.

This example gives us our definition of the next set of the polynomial hierarchy we are examining, $\Sigma_1$, also known as NP. A language, $L$, of words, $W$, is in $\Sigma_1$ if and only if there exists some language, $L' \in P$, of pairs, $(W,V)$, where a word, $W$ is in $L$ if and only if there exists some $V$ such that $(W,V) \in L'$. Therefore, since $L_S'$ is in $P$ we see that $L_S$ is in $\Sigma_1$.

Of course ``there exists'' is not the only quantifier we can consider. Our next example will be the language of tautological formulae, $L_T$, that is, those logical formulae for which every assignment of boolean variables results in the formula evaluating to true. Here we let $L_T'$ be the language consisting of words of the form $(F,A)$ where $F$ is a logical formula and $A$ can be anything except an assignment of variables to $F$ such that $F$ evaluates to false. Clearly a word, $F$, belongs to $L_T$ if and only if ``for all'' $A$ we have that $(F,A) \in L_T'$. Once again we have that $L_T'$ is in $P$ but there exists no known polynomial-time algorithm to verify that a formula belongs to $L_T$.

This example gives us our definition of the next set that we are examining, $\Pi_1$ (or co-NP), which can be thought of as being at the same level of the hierarchy as $\Sigma_1$ but distinct from it. A language, $L$, of words, $W$, is in $\Pi_1$ if and only if there exists some language, $L' \in P$, of pairs $(W,V)$, where a word, $W$ is in $L$ if and only if for all $V$ we have that $(W,V) \in L'$. Therefore, since $L_T'$ is in $P$ we see that $L_T$ is in $\Pi_1$.

We can construct further levels of the polynomial hierarchy by alternating quantifiers. For our purposes we need only construct one more set, $\Sigma_2$. A language, $L$, of words, $W$, is in $\Sigma_2$ if and only if there exists some language, $L' \in \Pi_1$, of pairs, $(W,V)$, where a word $W$ belongs to $L$ if and only if there exists a $V$ such that $(W,V) \in L'$.

We will now show that both \textsc{Transversal Contraction} and \textsc{Transversal Dual} are in $\Sigma_2$. To see this we first have the fact that if given the associated bipartite graph for a set system $(E,\mathcal{A})$, we may determine whether a subset of $E$ is a transversal or not in polynomial time \cite{oxley}.

We first examine \textsc{Transversal Contraction}. Let $M_T$ be a transversal matroid on the same ground set as $M/X$ for which we have the transversal representation. Let $L_T$ be the language of triples of the form $(M/X, M_T, B)$ where $B$ can be anything except a set which is a basis in one of $M/X$ or $M_T$ but not the other. Since we can check in polynomial time if $B$ is a basis in $M/X$ or $M_T$ we have that $L_T \in P$.

Now let $L_X$ be the language of pairs of the form $(M/X,M_T)$ where $M/X \cong M_T$. We have that a word, $(M/X, M_T)$ is in $L_X$ if and only if for all $B$ we have that $(M/X,M_T,B) \in L_T$. Therefore $L_X$ is in $\Pi_1$.

Finally we have the language, $L$, with words of the form $M/X$ where $M/X$ is a transversal matroid. We see that determining whether a word is in $L$ is equivalent to solving \textsc{Transversal Contraction}. A word, $M/X$, is in $L$ if and only if there exists some transversal matroid $M_T$ such that $M/X \cong M_T$, that is if $(M/X, M_T) \in L_X$. Since $L_X \in \Pi_1$ this means that $L \in \Sigma_2$.

For \textsc{Transversal Dual} the argument is the same except that instead of $M/X$ we use $M^*$. That is, for a given transversal matroid $M_T$ with representation, the language of triples of the form $(M^*,M_T, B)$ where $B$ is not a basis of only one of $M^*$ or $M_T$, is in $P$. Therefore, the language of pairs of the form $(M^*, M_T)$ where $M^* \cong M_T$, is in $\Pi_1$. Therefore, the language of words, $M^*$, for which an $M_T$ exists such that $M^* \cong M_T$, is in $\Sigma_2$.

This means that determining if a dual or minor of a transversal matroid is also transversal is in $\Sigma_2$, even if we are only contracting a single element. One might suspect that $L$ should be at least in $\Pi_1$ given Equation~\ref{betaeq}, as all one has to do to verify that $M$ is not a transversal matroid is to find a set, $X$, for which $\beta(X)<0$, however, even if such an $X$ is found there could be exponentially many cyclic flats that strictly contain $X$ and so verifying the $\beta$ value of $X$ in polynomial time is likely to require some other technique, leaving the problem in $\Sigma_2$.
\section{Results}

\subsection{Co-Transversal Matroids}

We begin this section by calculating the dual of \eqref{betaeq}. That is, we wish to obtain a function which will be non-zero on all sets of a matroid if and only if that matroid is co-transversal.

At first, we have
\[\beta^*(Y) = r^*(M) - r^*(Y) - \sum_{Z \in \mathcal{Z}(M^*): Y \subset Z} \beta^*(Z).\]

We can use \eqref{rank dual} to get
\[\beta^*(Y) = |E(M)|-r(M) - r(\overline{Y})-|Y|+r(M) - \sum_{Z \in \mathcal{Z}(M^*): Y \subset Z} \beta^*(Z).\]

We may cancel the $r(M)$ terms and combine the cardinality terms
\[\beta^*(Y) = |\overline{Y}| - r(\overline{Y}) - \sum_{Z \in \mathcal{Z}(M^*): Y \subset Z} \beta^*(Z).\]

We wish to replace $Y$ with $\overline{Y}$ in the range of the summation. To do this note that if $Y \subset Z$ then $\overline{Z} \subset \overline{Y}$, this gives us
\[\beta^*(Y) = |\overline{Y}| - r(\overline{Y}) - \sum_{Z \in \mathcal{Z}(M^*): \overline{Z} \subset \overline{Y}} \beta^*(Z).\]

The first two terms in the sum are of course the nullity of $\overline{Y}$ which is defined as $n(\overline{Y}) = |\overline{Y}|-r(\overline{Y})$, while in the summation term it would help if we just examined $\overline{Z}$ directly. This is fortunately, quite easy, as the cyclic flats of $M$ are exactly the complements of cyclic flats of $M^*$, hence we have
$$\beta^*(Y) = n(\overline{Y}) - \sum_{Z \in \mathcal{Z}(M): Z \subset \overline{Y}} \beta^*(\overline{Z}).$$

We now relable $\overline{Y}$ to be $X$ and then we have
\[\beta^*(\overline{X}) = n(X) - \sum_{Z \in \mathcal{Z}(M): Z \subset X} \beta^*(\overline{Z}).\]

Now we define a new function $\gamma(X)$ which is simply $\beta^*(\overline{X})$, we can then replace both of our $\beta^*$ terms as follows
\begin{equation}
\label{og gamma}
\gamma(X) = n(X) - \sum_{Z \in \mathcal{Z}(M): Z \subset X} \gamma(Z)
\end{equation}
and so we obtain our $\gamma$ function. Since the $\gamma$ function checks the same sets as the $\beta^*$ function we obtain the following Corollary.

\begin{corollary}
\label{gamma}
	A matroid, $M$, is co-transversal if and only if $\gamma(X)\geq 0$ for all $X \subseteq E(M)$.
\end{corollary}

Combining this result with Theorem~\ref{cotrans} implies the following:

\begin{corollary}
    A matroid, $M$, is a strict gammoid if and only if $\gamma(X) \geq 0$ for all $X \subseteq E(M)$.
\end{corollary}

The proof of Theorem~\ref{beta}, which can be found in \cite{jobonotes}, takes the cyclic flats with nonzero $\beta$ values and constructs a transversal representation. A similar construction can be obtained in the dual case for strict gammoids.

For this construction we define a few terms. If $M=M(D,V(D),S)$ is a strict gammoid we call $(D,V(D),S)$ a {\em representation} of $M$. We say that a representation is {\em maximal} if $(D+a,V(D),S)$ is not a representation of $M$ for any arc $a \notin A(D)$. If $v$ is a vertex of $V(D)-S$ then the set of vertices, $u$ for which the arc $(v,u) \in A(D)$ is called the {\em forward neighbourhood} of $v$, denoted $N_D^+(v)$. The {\em closed forward neighbourhood} of $v$, denoted $N_D^+[v]$, is $N_D^+(v)+v$. If $X$ is a set of vertices of $D$ then the set $N(X)$ is the set of vertices, $v$ such that $N^+_D[v]=X$.

\begin{theorem}
\label{without proof}
    Let $D_M$ be a maximal representation of a strict gammoid, $M$. Let $X$ be a subset of $E(M)$. Then $X$ is a closed forward neighbourhood of $D_M$ if and only if $X$ is a cyclic flat of $M$ with nonzero $\gamma$ value and $|N(X)|=\gamma(X)$.
\end{theorem}

Our version of the proof of Theorem~\ref{without proof} is quite detailed and involves many results about strict gammoids that are not used in this paper. For this reason the proof is ommitted and will be available in the author's thesis. We justify this with the fact that the dual result exists for transversal matroids as can be found in \cite{jobonotes} and the fact that we will not use Theorem~\ref{without proof} directly.

What this Theorem tells us is that we can simply read off the $\gamma$ values of cyclic flats with nonzero $\gamma$ values from the maximal directed graph representation. Similarly if we have these particular $\gamma$ values we can construct the maximal directed graph representation. This does not directly follow from the result, however, this is something that is also proved in the author's thesis. The full result is as follows:

\begin{lemma}
\label{construct gammoid}
    Let $M$ be any matroid and let $\mathcal{F}(M)$ be the collection of cyclic flats of $M$ with positive $\gamma$ values. There is a polynomial-time algorithm that either constructs a maximal representation of a strict gammoid, $M'$, with $\mathcal{F}(M')=\mathcal{F}(M)$ and $\gamma_{M'}(F)=\gamma_M(F)$ for all $F \in \mathcal{F}(M')$ or verifies that no such strict gammoid exists. If $M$ is a strict gammoid then $M=M'$.
\end{lemma}

Note that if $M$ is not a strict gammoid, $M'$ may still exist, in which case $M$ is not equal to $M'$. This will cause difficulties for our algorithm as we will construct such an $M'$ and then have to verify that $M=M'$, however, we are able to do so in polynomial time.

Since any representation obtained from Lemma~\ref{construct gammoid} has $|E(M)|$ vertices there are at most $|E(M)|$ closed forward neighbourhoods and so viewing a strict gammoid as a directed graph is equivalent (for the purposes of polynomial-time algorithms) to viewing the set of cyclic flats with nonzero $\gamma$ values and their $\gamma$ values. Therefore we use Theorem~\ref{without proof} to justify viewing strict gammoids by their collections of cyclic flats with nonzero $\gamma$ values rather than as directed graph representations.

To do this in general we need one more result which allows us to obtain a maximal representation of a strict gammoid from a general representation of a strict gammoid. Once again, the dual of this result exists for transversal matroids and can be found in \cite{jobonotes}. Therefore we will not provide a proof here and instead refer the reader to the author's thesis for details.

\begin{lemma}
\label{poly time maximality}
    Let $(D,V(D),S)$ be a representation of a strict gammoid $M$. There exists an algorithm which obtains a maximal representation for $M$ in $\mathcal{O}(|V(D)|^2)$ operations.
\end{lemma}

Finally we also explicitly state the result we proved and used above about the bound on the number of cyclic flats of a strict gammoid with nonzero $\gamma$ values.

\begin{corollary}
\label{bound corollary}
    Let $M$ be a strict gammoid. Then
    \[
    \sum_{Z \in \mathcal{Z}(M): Z \subset E(M)} \gamma(Z) \leq |E(M)|.
    \]
\end{corollary}

\subsection{Flats}

We require some results concerning flats of matroids. Some of these results pertain to general matroids rather than simply strict gammoids.

\begin{proposition}
\label{max flat prop}
    Let $X$ be a subset of a matroid, $M$ and let $L^*$ be the set of coloops of $M|X$. Then $X-L^*$ is the unique maximal cyclic set contained in $X$ and $n(X)=n(X-L^*)$.

    Furthermore if $X$ is a flat then $X-L^*$ is a cyclic flat.
\end{proposition}
\begin{proof}
    Removing $L^*$ from $X$ decreases the size and the rank of $X$ by the same amount so we see that $n(X)=n(X-L^*)$. If $X$ is closed $cl(X-L^*) \subseteq X$ and since adding any elements from $L^*$ would increase the rank of $X-L^*$ we see that $cl(X-L^*)=X-L^*$ and finally since $M|(X-L^*)$ contains no coloops $X-L^*$ is a cyclic flat.

    Also any other maximal cyclic set $X'$ could not contain any elements in $L^*$ since they are coloops of $M|X$ and therefore would be contained in $X-L^*$, hence $X-L^*=X'$ and therefore $X-L^*$ is unique.
\end{proof}

We can use Proposition~\ref{max flat prop} to see that in a strict gammoid, $M$, we have

\begin{equation}
\label{bound}
n(M) = \sum_{Z \in \mathcal{Z}(M) } \gamma(Z).
\end{equation}

Which we can see by letting $Z_\omega$ be the maximum cyclic flat contained in $E(M)$ described by Proposition~\ref{max flat prop}, then
\[\gamma(Z_\omega) = n(Z_\omega) - \sum_{Z \in \mathcal{Z}(M)-Z_\omega} \gamma(Z)\]
but by Proposition~\ref{max flat prop} we see that $n(Z_\omega) = n(M)$ and so rearranging gives the result. This gives another proof of Corollary~\ref{bound corollary}. Next we have two results about the $\gamma$ function that we will use later.

\begin{lemma}
\label{find cyclic}
    Let $X$ be a non-cyclic set in any matroid and let $L^*$ be the set of coloops of $M|X$. If $X-L^*$ is not a cyclic flat then $\gamma(X) = \gamma(X-L^*)$, and if $X-L^*$ is a cyclic flat then $\gamma(X)=0$.
\end{lemma}

\begin{proof}
    First assume that $X-L^*$ is not a cyclic flat. Then, since no cyclic flats in $X$ can contain elements of $L^*$ we have that the collection of cyclic flats strictly contained in $X-L^*$ is the same as the collection of cyclic flats strictly contained in $X$. Therefore
    \[
    \sum_{Z \in \mathcal{Z}(M): Z \subset X} \gamma(Z) = \sum_{Z \in \mathcal{Z}(M): Z \subset X-L^*} \gamma(Z).
    \]

    Also by Proposition~\ref{max flat prop} we have that $n(X)=n(X-L^*)$ and therefore
    \[
    n(X) - \sum_{Z \in \mathcal{Z}(M): Z \subset X} \gamma(Z) = n(X-L^*) - \sum_{Z \in \mathcal{Z}(M): Z \subset X-L^*} \gamma(Z)
    \]
    which is of course
    \[
    \gamma(X)=\gamma(X-L^*)
    \]
    as required.

    On the other hand if $X-L^*$ is a cyclic flat then the collection of cyclic flats strictly contained in $X$, that is $Z \in \mathcal{Z}(M),Z \subset X$, contains one cyclic flat more than the collection of cyclic flats strictly contained in $X-L^*$, namely, $X-L^*$ is a cyclic flat strictly contained in $X$ but not in $X-L^*$. This means that
    \[
    (Z \in \mathcal{Z}(M):Z \subset X) = (Z \in \mathcal{Z}(M): Z \subset X-L^*) + (X-L^*)
    \]
    this gives us
    \[
    n(X) - \sum_{Z \in \mathcal{Z}(M): Z \subset X} \gamma(Z) = n(X-L^*) - \left( \sum_{Z \in \mathcal{Z}(M): Z \subset X-L^*} \gamma(Z) \right) - \gamma(X-L^*)
    \]
    \[
    \gamma(X) = \gamma(X-L^*) - \gamma(X-L^*) = 0
    \]
    as required.
\end{proof}

Note that if $X$ in Lemma~\ref{find cyclic} is a flat then we have from Proposition~\ref{max flat prop} that $X-L^*$ will be a cyclic flat and so we obtain the following corollary:

\begin{corollary}
\label{flat lemma}
    Let $X$ be a non-cyclic flat in any matroid. Then $\gamma(X)=0$.
\end{corollary}

\begin{lemma}
\label{containment lemma}
	Let $X$ be a flat in a strict gammoid. Then for every $x \in X$ such that $x$ is not a coloop of $M|X$, there is some cyclic flat $Z_x \subseteq X$ such that $x \in Z_x$ and $\gamma(Z_x)>0$.
\end{lemma}
\begin{proof}
	Let $x$ be an element of $X$ and let there be no cyclic flat $Z_x \subseteq X$ such that $x \in Z_x$ and $\gamma(Z_x)>0$. Clearly we have that if $X$ is a cyclic flat $\gamma(X)=0$ since $x \in X$. By the definition of $\gamma$ we have
	\[\gamma(X) = n(X) - \sum_{Z \in \mathcal{Z}(M): Z\subset X} \gamma(Z).\]

	If $X$ is cyclic then we can rearrange this as follows:
	\begin{equation} \label{eq1,part2}
		n(X) = \sum_{Z \in \mathcal{Z}(M): Z \subseteq X} \gamma(Z).
	\end{equation}

    If $X$ is not cyclic then by Proposition~\ref{max flat prop} there exists a unique maximal cyclic flat $Z_0 \subseteq X$ and $n(X)=n(Z_0)$. Since $Z_0$ is unique and maximal we see that all cyclic flats contained in $X$ are contained in $Z_0$, this gives us
	\[n(X) = n(Z_0) = \sum_{Z \in \mathcal{Z}(M): Z \subseteq X} \gamma(Z).\]

	Therefore \eqref{eq1,part2} still holds even though $X$ is not cyclic. But since $x$ is not a coloop of $M|X$ we have that $r(X- x) = r(X)$ and so
	\begin{equation}\label{eq2}
	n(X- x) = |X - x| - r(X- x) = |X|-r(X) - 1 = n(X)-1.
	\end{equation}	

    We now examine the sum of $\gamma$ values of cyclic flats contained in $X$. If $X$ is not a cyclic flat it is not in the sum and if it is a cyclic flat then it contains $x$ and so must have $\gamma$ value zero by assumption, either way we can remove it from the sum without changing the value of the sum.
    
    Similarly, every other cyclic flat contained in $X$ that contains $x$ must have $\gamma$ value of zero and so we can remove all of these cyclic flats from the sum as well, leaving only cyclic flats contained in $X-x$.

	\begin{equation}\label{eq3}	
	\sum_{Z \in \mathcal{Z}(M): Z\subseteq X} \gamma(Z) = \sum_{Z \in \mathcal{Z}(M): Z\subset X} \gamma(Z) = \sum_{Z \in \mathcal{Z}(M): Z\subset X- x} \gamma(Z)
	\end{equation}
	where we do not need $\subseteq$ for the last sum because if such a cyclic flat existed it would clearly contain $x$ in its closure. We now examine the $\gamma$ value of $X- x$ to get
	\[\gamma(X- x) = n(X- x) - \sum_{Z \in \mathcal{Z}(M): Z\subset X- x} \gamma(Z)\]
	and by \eqref{eq2} this is
	\[ \gamma(X- x) = n(X) - 1 - \sum_{Z \in \mathcal{Z}(M): Z\subset X- x} \gamma(Z)\]
	substituting in \eqref{eq1,part2} gives us
	\[\gamma(X- x) = \left( \sum_{Z \in \mathcal{Z}(M): Z \subseteq X} \gamma(Z) \right) - 1 - \sum_{Z \in \mathcal{Z}(M): Z\subset X- x} \gamma(Z)\]
	then we use \eqref{eq3} to get
	\[\gamma(X- x) =\left( \sum_{Z \in \mathcal{Z}(M): Z\subset X- x} \gamma(Z) \right) - 1 - \sum_{Z \in \mathcal{Z}(M): Z\subset X- x} \gamma(Z)\]
	cancelling gives
	\[\gamma(X- x) = -1\]
	which contradicts the fact that our matroid was a strict gammoid, giving us the result.
\end{proof}

\subsection{The Lattice of Cyclic Flats}

By Definition~\ref{cyclic flat axioms} we see that the collection of cyclic flats of a matroid forms a lattice under inclusion. Specifically if $Z_0$ and $Z_1$ are cyclic flats then their join, $Z_0 \lor Z_1$, is the closure of their union, and their meet, $Z_0 \land Z_1$ is the union of all circuits contained in their intersection. Note that the set of all loops of a matroid is a cyclic flat even if it is empty. If given a set of cyclic flats $\mathcal{Z} = \{Z_0, Z_1, \ldots , Z_n\}$ we define $\lor \mathcal{Z}$ to be $Z_0 \lor Z_1 \lor \dots \lor Z_n$ and $\land \mathcal{Z}$ to be $Z_0 \land Z_1 \land \dots \land Z_n$. It is clear that given any pair of cyclic flats in a gammoid their join and meet can be obtained in polynomial time, however, there may be exponentially many cyclic flats in the matroid and so it is not feasible to examine all of them in polynomial time.

Our goal is to examine a single-element deletion from a strict gammoid and determine if the resulting gammoid remains a strict gammoid. Naively in order to do this we would have to check the $\gamma$ values of every subset of the ground set of the resulting matroid which clearly could not be done in polynomial time. In fact, even checking all the cyclic flats could not be done in polynomial time. However, recall that by Corollary~\ref{bound corollary} the number of cyclic flats with nonzero $\gamma$ values in the original matroid is bounded by the size of the ground set. While that may not be true for our new matroid we can use this information to help us obtain information about it in polynomial time. As such we first examine the relationship between the lattices of cyclic flats for each matroid.

Let $M^+$ be a matroid with $e \in E(M^+)$, let $M=M^+ \backslash e$. We obtain a function $f: \mathcal{Z}(M^+) \rightarrow 2^{E(M)}$ which is defined as $f(Z) = Z-e$. Note that $f(Z)$ may or may not be a cyclic flat in $M$. We now prove the following properties about $f$.

\begin{lemma}
\label{f lemma}
	The function $f$ has the following properties for all $Z_0, Z_1 \in \mathcal{Z}(M^+)$ and $Z \in \mathcal{Z}(M)$:
	\begin{enumerate}[label=\ref{f lemma}.\arabic*.,ref=\ref{f lemma}.\arabic*]
		\item \label{f lemma.1} If $f(Z_0) = f(Z_1)$ then $Z_0 = Z_1$.
		\item \label{f lemma.2} $cl_M(f(Z_0) \cup f(Z_1)) = f(cl_{M^+}(Z_0 \cup Z_1)) $
		\item \label{f lemma.3} Let $L^*$ be the set of all coloops of $M|(f(Z_0) \cap f(Z_1))$. We have that 
        \[(f(Z_0) \cap f(Z_1))-L^* \subseteq f(Z_0 \land Z_1)\] 
        and if $f(Z_0 \land Z_1)$ is cyclic in $M$ then 
        \[(f(Z_0) \cap f(Z_1))-L^* = f(Z_0 \land Z_1).\]
		\item \label{f lemma.4} There exists some $Z' \in \mathcal{Z}(M^+)$ such that $f(Z') = Z$.
		\item \label{f lemma.5} Either $f(Z_0) \in \mathcal{Z}(M)$ or $\gamma_M(f(Z_0))=0$.
	\end{enumerate}
\end{lemma}
\begin{proof}
	For the first result note that if $f(Z_0) = f(Z_1)$ and $Z_0 \neq Z_1$ then the symmetric difference of $Z_0$ and $Z_1$ would be the singleton element $\{e\}$ which is impossible for cyclic flats \cite{oxley}.

	For the second result note that since $Z_0$ and $Z_1$ are both cyclic flats in $M^+$ we have that $e$ is not a coloop of either $M^+|Z_0$ or $M^+|Z_1$ and therefore clearly isn't a coloop of $M^+|Z_0 \cup Z_1$. Hence $r_{M^+}((Z_0-e) \cup(Z_1 -e)) = r_{M^+}(Z_0 \cup Z_1)$. Hence $cl_{M^+}(Z_0 \cup Z_1) = cl_{M^+}((Z_0 - e) \cup (Z_1 - e))$. Which means that
    \[
    cl_{M^+}(Z_0 \cup Z_1) = cl_{M^+}(f(Z_0) \cup f(Z_1)).
    \]
    Since $cl_{M^+}(Z_0 \cup Z_1)$ is a cyclic flat of $M^+$ we can apply $f$ to it to get
    \[
    f(cl_{M^+}(Z_0 \cup Z_1)=cl_{M^+}(f(Z_0) \cup f(Z_1))-e.
    \]
    However, since $e \notin f(Z_0) \cup f(Z_1)$ we have that $cl_{M^+}(f(Z_0) \cup f(Z_1))-e=cl_M(f(Z_0) \cup f(Z_1))$ and therefore we obtain the desired equality.

    For the third result; any element, $e_0$, in $f(Z_0) \cap f(Z_1)$ that is not in $L^*$ must be in a circuit, $C_0$ of $M$, contained in $f(Z_0) \cap f(Z_1)$. Any circuit of $M$ is also a circuit of $M^+$ and so $e_0$ is not a coloop of $M^+|(f(Z_0) \cap f(Z_1))$ and therefore also not a coloop of $M^+|(Z_0 \cap Z_1)$. Hence $e_0 \in (Z_0 \land Z_1)$ and therefore also in $f(Z_0 \land Z_1)$ as required.

    If $f(Z_0 \land Z_1)$ is cyclic in $M$ then any element $e_0 \in f(Z_0 \land Z_1)$ must be contained in a circuit $C_0$ in $M$ that is contained in $f(Z_0 \land Z_1)$. Since $f(Z_0 \land Z_1) \subseteq (Z_0 \cap Z_1)-e$ we see that $C_0 \subseteq (Z_0 \cap Z_1)-e$. Since $(Z_0\cap Z_1)-e = (Z_0-e) \cap (Z_1-e)$ we see that $C_0$ is in both $f(Z_0)$ and $f(Z_1)$. Hence $e_0$ cannot be in $L^*$ and so $e_0 \in f(Z_0) \cap f(Z_1)-L^*$ as required.

	For the fourth result we find $Z'$ by taking $cl_{M^+}(Z)$. Since every element in $cl_{M^+}(Z)-Z$ must complete a circuit with elements in $Z$ that element would complete that same circuit with those elements in $M$ and therefore the only element that can be in $cl_{M^+}(Z)-Z$ is $e$. Hence $cl_{M^+}(Z)=Z'$.

    For the last result we have that $f(Z_0)$ is a flat of $M$ from \cite[Proposition 3.3.7(ii)]{oxley} and therefore the result follows from Corollary~\ref{flat lemma}.
\end{proof}

Let $\mathcal{L}^+$ be the lattice of cyclic flats of $M^+$ and let $\mathcal{L}$ be the lattice of cyclic flats of $M$. Lemma~\ref{f lemma} tells us that $f$ is almost an isomorphism from $\mathcal{L}^+$ and $\mathcal{L}$. The `almost' coming from the fact that sometimes $f$ sends cyclic flats to sets that are not cyclic flats and we have $f(Z_0)\land f(Z_1) \subseteq f(Z_0 \land Z_1)$ rather than $f(Z_0)\land f(Z_1) = f(Z_0 \land Z_1)$. We will use this near-isomorphic structure to obtain the results we need.

\begin{lemma}
\label{first lemma}
	Let $X$ be a cyclic set in $M^+$ such that $e \in X$. Then
	\[\gamma_{M^+}(X)  - \gamma_M(X-e) -1 = \sum_{Z \in \mathcal{Z}(M^+): Z \subset X} (\gamma_M(Z-e) - \gamma_{M^+}(Z)).\]
 
	On the other hand if $e \notin X$ and $X$ is a flat of $M^+$ then $\gamma_{M^+}(X)-\gamma_M(X-e)=0$.
\end{lemma}
\begin{proof}
	To see the first result we begin with the following equation derived from the definition of the $\gamma$ function:
    \begin{multline*}
	\gamma_{M^+}(X) - \gamma_M(X-e) = n_{M^+}(X) - \left( \sum_{Z \in \mathcal{Z}(M^+): Z \subset X} \gamma_{M^+}(Z) \right)\\ 
    - n_M(X-e) +\sum_{Z \in \mathcal{Z}(M): Z \subset X-e} \gamma_M(Z).
    \end{multline*}

	But since $e \in X$ and $X$ is cyclic in $M^+$ we see that $e$ is not a coloop of $M^+|X$, hence $r_{M^+}(X) = r_{M^+}(X- e)$ but of course $r_{M^+}(X-e) = r_M(X-e)$ and so $r_{M^+}(X) = r_M(X-e)$ and therefore $n_{M^+}(X)-1 = n_M(X-e)$. Hence we have
	\[\gamma_{M^+}(X) - \gamma_M(X-e) = 1- \sum_{Z \in \mathcal{Z}(M^+): Z \subset X} \gamma_{M^+}(Z) + \sum_{Z \in \mathcal{Z}(M): Z \subset X-e} \gamma_M(Z).\]
 
	We now wish to combine the two sums into one. In order to do this we must first change the range of the second sum from $Z \in \mathcal{Z}(M), Z \subseteq X-e$ to $Z \in \mathcal{Z}(M^+),Z \subseteq X$. Obviously these terms are not necessarily cyclic flats of $M$ so the terms of the sum will become $\gamma_M(Z-e)$.

    We need to show that changing the sum in this way leaves its total value the same. First we see by Lemma~\ref{f lemma.4} that we have not lost any terms and so we need only worry about adding additional terms. We see from Lemma~\ref{f lemma.1} that we are not adding any terms that we already have. Therefore we must only be adding terms that are not in $Z \in \mathcal{Z}(M): Z \subset X-e$.

    We are changing two parts to this range. First we are expanding from looking at cyclic flats of $M$ to cyclic flats of $M^+$. In doing so we may add terms which are not cyclic flats of $M$. However, by Lemma~\ref{f lemma.5} all of these will have $\gamma_M$ value of 0 and so adding them does not change the sum.

    We are also expanding our range from $Z \subset X-e$ to $Z \subset X$. The only new subset of $E(M)$ that this change could potentially give us is $X-e$ itself. However since $X$ is a cyclic set of $M^+$ we must have that $e \in cl_{M^+}(X-e)$ and so $X-e$ is not a cyclic flat of $M^+$, therefore we are not adding $X-e$ to the range of the sum.

	This means that all the terms we are adding have $\gamma_M$ values of 0 and so do not change the sum, this allows us to obtain.
	\[\gamma_{M^+}(X) - \gamma_M(X-e) = 1- \sum_{Z \in \mathcal{Z}(M^+): Z \subset X} \gamma_{M^+}(Z) + \sum_{Z \in \mathcal{Z}(M^+): Z \subset X} \gamma_M(Z-e)\]
	\[\gamma_{M^+}(X) - \gamma_M(X-e) -1 = \sum_{Z \in \mathcal{Z}(M^+): Z \subset X} ( \gamma_M(Z-e) - \gamma_{M^+}(Z) )\]
	as required. For the second result note that for any set $X' \subseteq X$ we have that $n_M(X')=n_{M^+}(X')$ since no elements have been removed from $X'$. Similarly, for any flat $F \subseteq X$ in $M^+$ we have that $F$ is still a flat in $M$ \cite{oxley} and for any cyclic flat $Z \not \subseteq X$ in $M^+$ we cannot have $Z-e \subseteq X$ as since $Z$ is cyclic $e \in cl_{M^+}(Z-e)$ and therefore $e$ would be in $cl_{M^+}(X)$, contradicting the fact that $X$ is a flat that does not contain $e$.

    Hence all cyclic flats in $X$ are the same in both $M$ and $M^+$ and all have the same nullities, as does $X$ itself. Hence $\gamma_{M^+}(X)=\gamma_M(X)$.
\end{proof}

To simplify notation we will use $\Delta \gamma(X)$ to mean $\gamma_M(X-e)-\gamma_{M^+}(X)$ for any $X \subseteq E(M^+)$.
Also for a collection $\mathcal{X}$ of sets $X \subset E(M^+)$ we will use $\down(\mathcal{X})$ to mean the collection of cyclic flats of $M^+$ that are each contained in at least one of the sets $X \in \mathcal{X}$. Note that all cyclic flats $Z \in \mathcal{X}$ are contained in $\down(\mathcal{X})$. If $\mathcal{X}$ consists of a single set $X$ we simplify notation by writing $\down(X)$ instead of $\down(\{X\})$. This allows us to rewrite the equation from Lemma~\ref{first lemma} as

\begin{equation} \label{my equation :)}
    -\Delta \gamma(X)-1 = \sum_{Z \in \down(X)-X} \Delta \gamma(Z).
\end{equation}

We use this notation for the next results:

\begin{lemma}
\label{meet lemma}
	Let $X$ be a cyclic set in $M^+$ with $e \in X$. Let $\mathcal{Y}$ be a non-empty collection of cyclic flats in $M^+$, each properly contained in $X$ such that $e \in \land \mathcal{Y}$. Then
	$$-\Delta \gamma(X) = \sum_{\substack{Z \in \down(X)-X\\ Z \notin \down(\mathcal{Y}) }} \Delta \gamma(Z)$$
\end{lemma}
\begin{proof}
	We prove this by contradiction. Assume that the result is false and that we have chosen $M^{+}$, $e$, $X$, and $\mathcal{Y}$ amongst all counterexamples so that $|\mathcal{Y}|$ is as small as possible. Since $\mathcal{Y}$ is non-empty this means that we first examine what happens when $|\mathcal{Y}|=1$. To do this we start with \eqref{my equation :)}
	$$-\Delta \gamma(X)  -1 = \sum_{Z \in \down(X)-X} \Delta \gamma(Z)$$ 
    and then let $\mathcal{Y}$ be just one cyclic flat $Y \in \mathcal{Z}(M^+)$. We can then take out the terms in the sum that are contained within $Y$.

	$$-\Delta \gamma(X) - 1 = \sum_{\substack{Z \in \down(X)-X\\Z \notin \down(Y)}} \Delta \gamma(Z) + \left( \sum_{Z \in \down(Y)-Y} \Delta \gamma(Z) \right) + \Delta \gamma (Y)$$
	but since $e \in Y$ we can use \eqref{my equation :)} again, replacing the second summation with a negative $\Delta \gamma(Y)$ term
	$$-\Delta \gamma(X) - 1 = \left( \sum_{\substack{Z \in \down(X)-X\\Z \notin \down(Y)}} \Delta \gamma(Z) \right) -\Delta \gamma(Y) - 1 + \Delta \gamma (Y)$$
	we then cancel to obtain
	$$-\Delta \gamma(X) = \sum_{\substack{Z \in \down(X)-X\\Z \notin \down(Y)}} \Delta \gamma(Z)$$
	which contradicts the fact that $\mathcal{Y}$ provides a counterexample. We therefore let $|\mathcal{Y}| =k>1$. Let $Y_k$ be a maximal element of $\mathcal{Y}$ and note that since $e\in \land \mathcal{Y}$ we must have that $e \in \land(\mathcal{Y}-Y_k)$ as well and therefore by our minimality assumption we have
	$$-\Delta \gamma(X) = \sum_{\substack{Z \in \down(X)-X\\Z \notin \down(\mathcal{Y}-Y_k) }} \Delta \gamma(Z)$$
	but once again we can take out the terms in the sum that are in $\down(Y_k)$ to get
	$$-\Delta \gamma(X) = \sum_{\substack{Z \in \down(X)-X\\Z \notin \down(\mathcal{Y}) }} \Delta \gamma(Z) + \left( \sum_{\substack{Z \in \down(X)-X\\ Z \notin \down(\mathcal{Y}-Y_k) \\ Z \in \down(Y_k)-Y_k }} \Delta \gamma(Z) \right) + \Delta \gamma(Y_k)$$
	however there are various ways to rewrite the range of the second summation term. First note that since $Y_k \subset X$ we do not need the $Z \in \down(X)-X$ condition. Next we construct the collection $\mathcal{Y}'$ by adding $Y \land Y_k$ for each $Y \in (\mathcal{Y}-Y_k)$. We may replace the $Z \notin \down(\mathcal{Y}-Y_k)$ condition with $Z \notin \down(\mathcal{Y}')$ instead as for each set $Y \in \mathcal{Y}-Y_k$ in order for an element to be excluded from the second sum by $Y$ it would have to be in both $\down(Y)$ and $\down(Y_k)$, therefore meaning it would have to be in $\down(Y \land Y_k)$ and therefore in $\down(\mathcal{Y}')$. Hence for this particular sum $Z \notin \down(\mathcal{Y}')$ provides the same restriction as $Z \notin \down(\mathcal{Y}-Y_k)$ does. Hence we can rewrite our equation as
	$$-\Delta \gamma(X) = \sum_{\substack{Z \in \down(X)-X\\Z \notin \down(\mathcal{Y}) }} \Delta \gamma(Z) + \left( \sum_{\substack{Z \in \down(Y_k)-Y_k \\ Z \notin \down(\mathcal{Y}') }} \Delta \gamma(Z) \right) + \Delta \gamma(Y_k)$$
    but then we see that the $\land$ operation is associative and commutative we have that $\land \mathcal{Y}' = \land \mathcal{Y}$ and therefore $e \in \land \mathcal{Y}'$ since $e \in \land{\mathcal{Y}}$. We also have that $e \in Y_k$. Also since $Y_k$ is a maximal element of $\mathcal{Y}$ each element $Y' \in \mathcal{Y}'$ is strictly contained in $Y_k$. We also have that $\mathcal{Y}'$ is non-empty since $|\mathcal{Y}|>1$. And finally we have that $|\mathcal{Y}'|$ is at most $k-1$, hence $Y_k$ and $\mathcal{Y}'$ fulfill all of our conditions and therefore by our minimality assumption we have
	$$-\Delta \gamma(X) = \left( \sum_{\substack{Z \in \down(X)-X\\Z \notin \down(\mathcal{Y}) }} \Delta \gamma(Z) \right) -\Delta \gamma(Y_k) + \Delta \gamma(Y_k)$$
	which, of course, cancels to give us
	$$-\Delta \gamma(X) = \sum_{\substack{Z \in \down(X)-X\\Z \notin \down(\mathcal{Y}) }} \Delta \gamma(Z) $$
	contradicting our assumption and therefore proving the result for all $k$.
\end{proof}

\begin{lemma}
\label{hooray corollary}
	Let $X$ be a cyclic set in $M^+$ with $e \in X$. Let there be some cyclic flat $Z_X$ of $M^+$ such that $Z_X \subset X$ and $\Delta \gamma(Z_X)$ is positive. Let $Z_0 \lor Z_1 \subset X$ for cyclic flats $Z_0$ and $Z_1$ of $M^+$, whenever $Z_0 \subset X$ and $Z_1 \subset X$ and, both $\Delta \gamma(Z_0)$ and $\Delta \gamma(Z_1)$ are positive. Then $\Delta \gamma(X) \geq 0$.

	Similarly let $W$ be a cyclic set in $M^+$ with $e \in W$. Let $Z_0 \lor Z_1 \subset W$ for cyclic flats $Z_0$ and $Z_1$ of $M^+$, whenever $Z_0 \subset W$ and $Z_1 \subset W$ and, both $\Delta \gamma(Z_0)$ and $\Delta \gamma(Z_1)$ are negative. Then $\Delta \gamma(W) \leq 0$.
\end{lemma}
\begin{proof}
    We first examine $X$. Note that our hypothesis is non-trivial as if $X$ is not a flat it is possible for $Z_0 \lor Z_1$ to not be contained in $X$ and even if $X$ is a cyclic flat we could still have $Z_0 \lor Z_1 = X$ in which case $Z_0 \lor Z_1$ would not be properly contained in $X$ as we require.

    Let $\mathcal{Y}$ be the collection of cyclic flats $Y$ in $M^+$ such that $Y \subset X$ and $\Delta \gamma(Y)>0$. Since $\Delta \gamma(Y) \neq 0$ we have from Lemma~\ref{first lemma} that $e \in Y$ for all $Y \in \mathcal{Y}$. Note that since there exists some cyclic flat $Z_X$ of $M^+$ such that $Z_X \subset X$ and $\Delta \gamma(Z_X)>0$ we have that $\mathcal{Y}$ is non-empty.

	Let $Y'$ be some cyclic flat of $\mathcal{Y}$ and let $\mathcal{Y}'$ be the set $\{ Y' \lor Y, \forall Y \in \mathcal{Y}-Y'\}$ if $|\mathcal{Y}|>1$, otherwise let $\mathcal{Y}'=\{Y'\}$. Either way note that since $\mathcal{Y}$ is non-empty we have that $\mathcal{Y}'$ is non-empty as well. We also have that since every cyclic flat of $\mathcal{Y}$ contains $e$ every cyclic flat of $\mathcal{Y}'$ contains $e$ as they all contain at least one cyclic flat of $\mathcal{Y}$. Note also that by hypothesis every cyclic flat of $\mathcal{Y}'$ is strictly contained in $X$ as they are all either joins of cyclic flats strictly contained in $X$ with positive $\Delta \gamma$ values, or simply cyclic flats that are already strictly contained in $X$. Finally we have that $Y' \subseteq \land \mathcal{Y}'$ and therefore $e \in \land \mathcal{Y}'$ and so by Lemma~\ref{meet lemma} we have
	$$-\Delta \gamma(X) = \sum_{\substack{Z \in \down(X)-X\\ Z \notin \down(\mathcal{Y}') }} \Delta \gamma(Z)$$
 
	However, since every cyclic flat of $\mathcal{Y}$ is contained within one or more cyclic flats of $\mathcal{Y}'$ we have that this sum contains no cyclic flats of $\mathcal{Y}$, nor any cyclic flats contained within them. Hence for any $Z$ in this sum $\Delta \gamma(Z)$ must be nonpositive. Hence all terms in the sum are nonpositive and so
	$$-\Delta \gamma(X) \leq 0$$
	which of course means
	$$\Delta \gamma(X) \geq 0$$
	as required.

	The second result is similar except we do not require the $Z_X$ condition. This is because of the -1 term from \eqref{my equation :)} which makes $\Delta \gamma(W) \leq 0$ when there are no cyclic flats strictly contained in $W$ with negative $\Delta \gamma$ values. Let $\mathcal{Y}$ be the collection of cyclic flats $Y$ in $M^+$ such that $Y \subset W$ and $\Delta \gamma(Y)<0$. Since $\Delta \gamma(Y) \neq 0$ we have from Lemma~\ref{first lemma} that $e \in Y$ for all $Y \in \mathcal{Y}$. We have two cases to consider, the first case is when $\mathcal{Y}$ is non-empty.

	Let $Y'$ be some cyclic flat of $\mathcal{Y}$ and let $\mathcal{Y}'$ be the set $\{Y' \lor Y, \forall Y \in \mathcal{Y}-Y\}$ if $|\mathcal{Y}|>1$, otherwise let $\mathcal{Y}'=\{Y'\}$. Either way note that since $\mathcal{Y}$ is non-empty we have that $\mathcal{Y}'$ is non-empty as well. We also have that since every cyclic flat of $\mathcal{Y}$ contains $e$ every cyclic flat of $\mathcal{Y}'$ contains $e$ as they all contain at least one cyclic flat of $\mathcal{Y}$. Note also that by hypothesis every cyclic flat of $\mathcal{Y}'$ is strictly contained in $W$ as they are all either joins of cyclic flats strictly contained in $W$ with negative $\Delta \gamma$ values, or simply cyclic flats that are already strictly contained in $W$. Finally we have that $Y' \subseteq \land \mathcal{Y}'$ and therefore $e \in \land \mathcal{Y}'$ and so by Lemma~\ref{meet lemma} we have
	$$-\Delta \gamma(W) = \sum_{\substack{Z \in \down(X)-X\\ Z \notin \down(\mathcal{Y}') }} \Delta \gamma(Z)$$

	However, since every cyclic flat of $\mathcal{Y}$ is contained within one or more cyclic flats of $\mathcal{Y}'$ we have that this sum contains no cyclic flats of $\mathcal{Y}$, nor any cyclic flats contained within them. Hence for any $Z$ in this sum $\Delta \gamma(Z)$ must be nonnegative. Hence all terms in the sum are nonnegative and so
	$$-\Delta \gamma(X) \geq 0$$
	which of course means
	$$\Delta \gamma(X) \leq 0$$
	as required. On the other hand if $\mathcal{Y}$ is empty then we have from Lemma~\ref{first lemma}
	$$-\Delta \gamma(W) - 1 = \sum_{Z \in \down(W)-W} \Delta \gamma(Z)$$
	however once again every term in the sum is nonnegative and so
	$$-\Delta \gamma(W) - 1 \geq 0$$
	$$\Delta \gamma(W) \leq 0$$
	as required.
\end{proof}

It will be more useful to view Corollary~\ref{hooray corollary} via its contrapositive.

\begin{corollary}
\label{contra}
    Let $X$ be a cyclic set in $M^+$ for which $\Delta \gamma(X)<0$ and there is some cyclic flat $Z_X$ of $M^+$ strictly contained in $X$ for which $\Delta \gamma (Z_X)>0$. Then there exists cyclic flats, $Z_0$ and $Z_1$ of $M^+$, both strictly contained in $X$ for which $\Delta \gamma(Z_0)>0$ and $\Delta \gamma(Z_1)>0$ and $Z_0 \lor Z_1$ is not strictly contained in $X$.

    Similarly let $W$ be a cyclic set in $M^+$ for which $\Delta \gamma(X)>0$. Then there exists cyclic flats, $Z_0$ and $Z_1$ of $M^+$, both strictly contained in $W$ for which $\Delta \gamma(Z_0)<0$ and $\Delta \gamma(Z_1)<0$ and $Z_0 \lor Z_1$ is not strictly contained in $W$.
\end{corollary}

Of course when $X$ and $W$ are cyclic flats any join of two cyclic flats that are both strictly contained within them is contained within them and so we obtain another more specific corollary for this case.

\begin{corollary}
\label{cyclic contra}
    Let $X$ be a cyclic flat of $M^+$ for which $\Delta \gamma(X)<0$ and there is some cyclic flat $Z_X$ of $M^+$ strictly contained in $X$ for which $\Delta \gamma(Z_X)<0$. Then there exists cyclic flats, $Z_0$ and $Z_1$ of $M^+$, both strictly contained in $X$ for which $\Delta \gamma(Z_0)>0$ and $\Delta \gamma(Z_1)>0$ and $Z_0 \lor Z_1 = X$.

    Similarly let $W$ be a cyclic flat of $M^+$ for which $\Delta \gamma(X)>0$. Then there exists cyclic flats, $Z_0$ and $Z_1$ of $M^+$, both strictly contained in $W$ for which $\Delta \gamma(Z_0)<0$ and $\Delta \gamma(Z_1)<0$ and $Z_0 \lor Z_1 = W$.
\end{corollary}

\subsection{Constructing the Algorithm}

We now assume that $M^+$ is a strict gammoid and describe a polynomial-time algorithm that determines whether $M$ is also a strict gammoid. Recall that $M=M^+\backslash e$. We will do this by first checking whether there are any cyclic flats of $M$ with negative $\gamma_M$ values and if there aren't we will check the $\gamma_M$ values of all other sets in $M$ by examining a few important sets. To check the $\gamma_M$ values of cyclic flats in $M$ we let first let $\mathcal{F}$ be the set of all cyclic flats, $Z$, of $M^+$ for which $\gamma_{M^+}(F) \neq 0$. Note these values must be positive since $M^+$ is a strict gammoid and also note that by Theorem~\ref{without proof} we can simply examine the closed forward neighbourhoods of our graphical representation to obtain these in polynomial time. Recall that $|\mathcal{F}|$ is bounded by $n=|E(M^+)|$ by \eqref{bound}.

Next we let $\mathcal{F}'$ be the collection of joins of any two cyclic flats in $\mathcal{F}$. Note that there are at most $n(n-1)/2$ cyclic flats in $\mathcal{F}'$ and so there are $\mathcal{O}(n^2)$ cyclic flats in $\mathcal{F} \cup \mathcal{F}'$. Finally we let $\mathcal{F}''$ be the collection of joins of any two cyclic flats in $\mathcal{F}'$. Note that for similar reasons there are $\mathcal{O}(n^4)$ cyclic flats in $\mathcal{F} \cup \mathcal{F}' \cup \mathcal{F}''$. Let $\mathcal{E}$ be the collection of cyclic flats $Z$ of $M$ for which there is some $F \in \mathcal{F} \cup \mathcal{F}' \cup \mathcal{F}''$ such that $Z=F-e$. This means that $|\mathcal{E}| \leq |\mathcal{F} \cup \mathcal{F}' \cup \mathcal{F}''|$ and so $|\mathcal{E}|$ is also of order $\mathcal{O}(n^4)$. The cyclic flats in $\mathcal{E}$ are the only cyclic flats that we will be examining, which will allow our algorithm to run in polynomial time. We define a new function $\eta$ on $M$, which works the same as the $\gamma_M$ function, except it only counts cyclic flats in $\mathcal{E}$. Formally, this is
\[
\eta(X) = n_M(X) - \sum_{\substack{Z \in \mathcal{Z}(M)\\ Z \subset X\\ Z \in \mathcal{E}}} \eta(Z)
\]
where sets that contain no such $Z$ have $\eta$ value equal to their nullity.

Our algorithm will work as follows:
\begin{itemize}
    \item For each $Z \in \mathcal{E}$ check that $\eta(Z) \geq 0$.
    \item If there exists a $Z \in \mathcal{E}$ with $\eta(Z) <0$ then $M$ is not a strict gammoid.
    \item Otherwise; let $\mathcal{F}_M$ be the set of cyclic flats of $M$ with positive $\gamma_M$ values (which we have obtained at this point by Theorem~\ref{algorithm theorem}).
    \item Check whether there is a strict gammoid $M'$ such that $\mathcal{F}_M$ is the set of cyclic flats of $M'$ with positive $\gamma$ values.
    \item If $M'$ does not exist then $M$ is not a strict gammoid.
    \item Otherwise; check that $r_M(Z_0 \cup Z_1) = r_{M'}(Z_0 \cup Z_1)$ and $cl_M(Z_0 \cup Z_1) = cl_{M'}(Z_0 \cup Z_1)$ for all $Z_0, Z_1 \in \mathcal{F}_M$.
    \item If there exists $Z_0$ and $Z_1$ in $\mathcal{F}_M$ such that either $r_M(Z_0 \cup Z_1) \neq r_{M'}(Z_0 \cup Z_1)$ or $cl_M(Z_0 \cup Z_1) \neq cl_{M'}(Z_0 \cup Z_1)$ then $M$ is not a strict gammoid.
    \item Otherwise; $M$ is a strict gammoid and $M=M'$.
\end{itemize}

Since $\mathcal{E}$ has order $\mathcal{O}(n^4)$ and checking the $\eta$ value of any cyclic flat in $\mathcal{E}$ only requires checking the $\eta$ value of at most every other cyclic flat in $\mathcal{E}$ we have that the first step of the this algorithm takes at most order $\mathcal{O}(n^8)$ time. While this is admittedly not very fast this is still polynomial time and is the slowest part of the algorithm.

If we find all the cyclic flats of $M$ with positive $\eta$ value, then we have found all cyclic flats of $M$ with positive $\gamma$ values as we show in Theorem~\ref{algorithm theorem}. We can then use these cyclic flats to attempt to construct a strict gammoid $M'$ as described by Lemma~\ref{construct gammoid}. If this construction fails then $M$ is not a strict gammoid and if $M$ is a strict gammoid then $M=M'$.

The last step requires checking ranks and closures of unions of the cyclic flats in $\mathcal{F}_M$ but since $\mathcal{F}_M$ is bounded by $n$ the order of sets to check is $\mathcal{O}(n^2)$ and so once again we can perform this step in polynomial time. We now verify that the algorithm will produce the correct answer.

\begin{theorem}
\label{algorithm theorem}
    If there exists a cyclic flat $Z$ of $M$ with $\gamma_M(Z)<0$ then there exists a cyclic flat $Z' \in \mathcal{E}$ with $\eta(Z')<0$. If every cyclic flat $Z$ of $M$ has $\gamma_M(Z) \geq 0$ then for all $Z' \in \mathcal{E}$ we have $\gamma_M(Z')=\eta(Z')$ and for all cyclic flats $Z''$ of $M$ with $Z'' \notin \mathcal{E}$ we have $\gamma_M(Z'')=0$.
\end{theorem}

\begin{proof}
    If $Z$ is a cyclic flat of $M$ then note that by Lemma~\ref{f lemma.1} there is a unique cyclic flat of $M^+$, which we will denote as $Z^+$ such that $Z^+ - e = Z$. If $Z \in \mathcal{E}$ then $Z^+ \in \mathcal{F} \cup \mathcal{F}' \cup \mathcal{F}''$.

    \begin{claim}
    \label{claim 1}
        Let $X$ be a cyclic flat of $M$ in $\mathcal{E}$. Then either $\eta(X) = \gamma_M(X)$ or there exists some cyclic flat $Y$ of $M$ with $Y \subset X$ such that $\gamma_M(Y)<0$.
    \end{claim}
    \begin{proof}
        To see this let $X$ be a minimal cyclic flat of $\mathcal{E}$ such that $\eta(X) \neq \gamma_M(X)$ and such that there does not exist any cyclic flat $Y$ of $M$ with $Y \subset X$ such that $\gamma_M(Y)<0$. First we examine the two equations for $\gamma_M(X)$ and $\eta(X)$. These are
        \[
        \gamma_M(X) = n_M(X) - \sum_{\substack{Z \in \mathcal{Z}(M)\\ Z \subset X}} \gamma_M(Z)
        \]
        and
        \[
        \eta(X) = n_M(X) - \sum_{\substack{Z \in \mathcal{Z}(M)\\ Z \subset X\\ Z \in \mathcal{E}}} \eta(Z).
        \]

        There are two ways that $\gamma$ and $\eta$ could differ. The first is that there is some $X' \in \mathcal{E}$ with $\gamma_M(X') \neq \eta(X')$, however, this would contradict the minimality of $X$. The second is that there is some $Z \notin \mathcal{E}$ with $\gamma_M(Z) \neq 0$. If neither of these occur then all the terms in both sums must be equal and there are no terms that are in one sum and not the other. Hence there must be some $Z \notin \mathcal{E}$ with $\gamma_M(Z) \neq 0$.

        Firstly if $\gamma_M(Z)<0$ then we simply set $Y=Z$ and we are done. Otherwise assume that $\gamma_M(Z)>0$. Since $Z \notin \mathcal{E}$ we must have that $Z^+ \notin \mathcal{F} \cup \mathcal{F}' \cup \mathcal{F}''$. This means $Z^+ \notin \mathcal{F}$ and so $\gamma_{M^+}(Z^+)=0$. Hence $\Delta \gamma(Z^+)>0$. Therefore, by Corollary~\ref{cyclic contra} there exist cyclic flats $Z_0$ and $Z_1$ in $M^+$, both strictly contained in $Z^+$, such that $\Delta \gamma(Z_0)<0$ and $\Delta \gamma(Z_1)<0$ and $Z^+= Z_0 \lor Z_1$. Then, since $Z^+$ is not in $\mathcal{F}'$ either, one of $Z_0$ or $Z_1$ is not in $\mathcal{F}$. Assume, without loss of generality, that $Z_0 \notin \mathcal{F}$. Then $\gamma_{M^+}(Z_0)=0$ and so since $\gamma_M(Z_0-e)-\gamma_{M^+}(Z_0)=\Delta \gamma(Z_0)<0$ we have that $\gamma_M(Z_0 -e) <0$. By Lemma~\ref{f lemma.5} this means that $Z_0 -e$ is a cyclic flat of $M$ and so we set $Y=Z_0 -e$ and we are done since $Z_0 - e \subset Z^+ -e = Z \subset X$.
    \end{proof}

    Next we prove the first statement of the theorem. To see this assume first that all $\eta$ values are nonnegative and let $Y$ be a cyclic flat of $M$ such that $\gamma_M(Y)<0$. We may assume $Y$ is minimal with the property that $\gamma_M(Y)<0$. If $Y \in \mathcal{E}$ then by Claim~\ref{claim 1} and the minimality of $Y$ we have that $\eta(Y) = \gamma_M(Y)$, a contradiction since all $\eta$ values are nonnegative. Hence we assume that $Y \notin \mathcal{E}$. In this case note that $Y^+ \notin \mathcal{F}$ and therefore $\gamma_{M^+}(Y^+)=0$ meaning that $\Delta \gamma(Y^+) < 0$, that is $\Delta \gamma(Y^+) \neq 0$. Therefore by Lemma~\ref{first lemma} we have that $e \in Y^+$ and so by Lemma~\ref{containment lemma} since $\gamma_{M^+}(Y^+)=0$ there is some cyclic flat $F_Y \in \mathcal{F}$ such that $F_Y \subset Y^+$ and $e \in F_Y$ (where we have strict containment because we know that $F_Y \neq Y^+$ since $Y^+ \notin \mathcal{F}$ and $F_Y \in \mathcal{F}$). We use Lemma~\ref{meet lemma} to get
    \[
    -\Delta \gamma(Y^+) = \sum_{\substack{Z \in \down(Y^+)-Y^+\\ Z \notin \down(F_Y)}} \Delta \gamma(Z)
    \]
    where since $\Delta \gamma(Y^+)<0$ there must be at least one $Z$ in this sum for which $\Delta \gamma(Z)>0$. Hence by Corollary~\ref{cyclic contra} again, there exist cyclic flats $Z_0$ and $Z_1$ in $M^+$, both strictly contained in $Y^+$ such that $\Delta \gamma(Z_0)>0$ and $\Delta \gamma(Z_1)>0$ and $Z_0 \lor Z_1 = Y^+$. We claim that both $Z_0$ and $Z_1$ are in $\mathcal{F}'$.

    Assume, without loss of generality, that $Z_0 \notin \mathcal{F}'$. Since $\Delta \gamma(Z_0)>0$ we have by Corollary~\ref{cyclic contra} that there exist cyclic flats $Z_2$ and $Z_3$ of $M^+$ both strictly contained in $Z_0$ such that $\Delta \gamma(Z_2)<0$ and $\Delta \gamma(Z_3)<0$ and $Z_2 \lor Z_3 = Z_0$. However, since $Z_0 \notin \mathcal{F}'$ it cannot be the join of two elements of $\mathcal{F}$ and so one of $Z_2$ or $Z_3$ is not in $\mathcal{F}$. We can assume, without loss of generality, that $Z_2 \notin \mathcal{F}$, and hence $\gamma_{M^+}(Z_2)=0$. However, since $\Delta \gamma(Z_2)<0$ this means that $\gamma_M(Z_2-e)<0$. Since $Z_2 -e \subset Z_0 -e \subset Y^+-e = Y$ this contradicts the minimality of $Y$. Hence we see that both $Z_0$ and $Z_1$ must be in $\mathcal{F}'$.

    This means that $Y^+ \in \mathcal{F}''$ which means $Y \in \mathcal{E}$ contradicting the fact that $Y \notin \mathcal{E}$. Hence we see that the algorithm will find negative $\eta$ values if any exist, otherwise it produces accurate $\gamma_M$ values for all the sets calculated by Claim~\ref{claim 1}. All that is left to show is that $\gamma_M(Z)=0$ for all cyclic flats $Z \notin \mathcal{E}$.

    To see this let $Y$ be a minimal cyclic flat of $M$ such that $\gamma_M(Y)>0$ but $Y \notin \mathcal{E}$. Then $Y^+ \notin \mathcal{F}$ and so $\gamma_{M^+}(Y^+)=0$ meaning $\Delta \gamma(Y)>0$ which means that, once again by Corollary~\ref{cyclic contra}, there exist cyclic flats $Z_0$ and $Z_1$ in $M^+$ such that both $\Delta \gamma(Z_0)<0$ and $\Delta \gamma(Z_1)<0$ and $Z_0 \lor Z_1 = Y^+$. Since all $\gamma_M$ values are nonnegative we know that for any cyclic flat $Z$ of $M^+$ if $\Delta \gamma(Z)<0$ we must have $\gamma_{M^+}(Z)>0$. Hence both $Z_0$ and $Z_1$ are in $\mathcal{F}$ and therefore $Y \in \mathcal{F}'$, a contradiction. Hence we obtain the result.
\end{proof}

Theorem~\ref{algorithm theorem} verifies that checking the $\eta$ values of sets in $\mathcal{E}$ is enough to tell us either that $M$ is not a strict gammoid, or what the $\gamma$ values of all cyclic flats in $M$ are. Since we only checked $\mathcal{E}$ there are at most $\mathcal{O}(n^4)$ cyclic flats of $M$ with positive $\gamma_M$ values. We first check these against Corollary~\ref{bound corollary} to ensure that there aren't too many. Then we can use Lemma~\ref{construct gammoid} to construct $M'$ from $\mathcal{F}_M$.

If $M'$ does not exist then we halt the algorithm and conclude that $M$ is not a strict gammoid. If $M$ is a strict gammoid then by Lemma~\ref{construct gammoid} we have that $M=M'$ so all that remains to check is whether or not $M=M'$.

Let the set of cyclic flats in $M$ with positive $\gamma_M$ values be $\mathcal{F}$. Since $M'$ is a strict gammoid the set of cyclic flats with positive $\gamma_{M'}$ values is precisely the set of closed forward neighbourhoods of $D$ which is precisely $\mathcal{F}$ by construction. What's more, each of these neighbourhoods has multiplicity exactly equal to its $\gamma_M$ value and therefore also equal to its $\gamma_{M'}$ value by Theorem~\ref{without proof}. Therefore for every $F$ in $\mathcal{F}$ we have that $F$ is a cyclic flat of both $M$ and $M'$ and $\gamma_M(F)=\gamma_{M'}(F)>0$. Since all other cyclic flats of $M$ have $\gamma_M$ value of 0 by the definition of $\mathcal{F}$ and all other cyclic flats of $M'$ have $\gamma_M'$ value of 0 by Theorem~\ref{without proof} and the fact that $M'$ has no sets with $\gamma_{M'}$ value less than 0 means that for all sets $X$ we have

\begin{equation} \label{F equation}
\sum_{Z \in \mathcal{Z}(M): Z \subset X} \gamma_M(Z) = \sum_{Z \in \mathcal{Z}(M'): Z \subset X} \gamma_M(Z).
\end{equation}

So if there exists some $X_M \subseteq E(M)$ for which $\gamma_M(X_M) \neq \gamma_{M'}(X_M)$ we must have that $n_M(X_M) \neq n_{M'}(X_M)$. We could then, of course, naively check the nullities of every set in both $M$ and $M'$ but this would, of course, be impossible to do in polynomial time. Instead we will locate a special set $X$ of $M$ with special properties.

\begin{lemma}
\label{final lemma}
    If $M$ is not a strict gammoid but every cyclic flat of $M$ has nonnegative $\gamma_M$ value then there exists a set $X$ of $M$ with negative $\gamma_M$ value such that $X$ contains two cyclic flats $Z_0$ and $Z_1$ of $M$, both with positive $\gamma_M$ values such that $X \subseteq (Z_0 \lor Z_1)$.
\end{lemma}
\begin{proof}
    We begin by examining the set $X_M$ which is a set of $M$ with negative $\gamma_M$ value. If $X_M$ is not cyclic then by Lemma~\ref{find cyclic} since $\gamma_M(X_M) \neq 0$ there exists a cyclic set contained in $X_M$ that also has negative $\gamma$ value. Therefore we may assume that $X_M$ is cyclic. We may therefore assume that $X_M$ is a maximal cyclic set of $M$ with a negative $\gamma_M$ value.

    \begin{claim}
    \label{Ze claim}
        There exists a cyclic flat $Z_e^+$ of $M^+$ such that $Z_e^+ \subset X_M+e$ and $e \in Z_e^+$.
    \end{claim}
    \begin{proof}
        Since $e \notin X_M$, for every cyclic flat, $Z^+$ of $M$, that is contained in $X_M$ we have that $\Delta \gamma(Z^+)=0$ by Lemma~\ref{first lemma}. Similarly since $e \notin X_M$ we have that $n_{M^+}(X_M)=n_M(X_M)$. Therefore
        \[
        n_{M^+}(X_M) - \sum_{Z \in \mathcal{Z}(M^+): Z \subset X_M} \gamma_{M^+}(Z) = n_M(X_M) - \sum_{Z \in \mathcal{Z}(M^+): Z \subset X_M} \gamma_M(Z-e).
        \]
        
        However, since $\gamma_M(X_M)<0$ and $\gamma_{M^+}(X_M) \geq 0$ we have that
        \[
        n_{M^+}(X_M) - \sum_{Z \in \mathcal{Z}(M^+): Z \subset X_M} \gamma_{M^+}(Z) \neq n_M(X_M) - \sum_{Z \in \mathcal{Z}(M): Z \subset X_M} \gamma_M(Z).
        \]

        So there must be some cyclic flat, $Z_e \in \mathcal{Z}(M)$ that is strictly contained in $X_M$ in $M$ but is not in $M^+$. This means that $Z+e=Z_e^+$ must be a cyclic flat of $M^+$ and therefore we obtain the result.
    \end{proof}

    Since $\gamma_{M^+}(X_M+e) \geq 0$ we have that $\Delta \gamma(X_M+e)<0$. By Claim~\ref{Ze claim} we have that there exists a cyclic flat, $Z_e^+ \subset X_M+e$ that contains $e$ which also means that $X_M+e$ is cyclic since $X_M$ is cyclic. Therefore by Lemma~\ref{meet lemma} we have
    \[
    -\Delta \gamma(X_M+e) = \sum_{\substack{Z \in \down(X_M+e)-(X_M+e)\\ Z \notin \down(Z^+_e)}} \Delta \gamma(Z).
    \]

    If all the terms in the sum are nonpositive then we contradict $\Delta \gamma(X_M+e) < 0$. Hence there exists some cyclic flat $Z_X^+$ of $M^+$ contained in $X_M+e$ such that $\Delta \gamma(Z_X^+)>0$. But then by Corollary~\ref{contra} there must exist cyclic flats $Z_0^+$ and $Z_1^+$ of $M^+$, both contained in $X_M+e$ with $\Delta \gamma(Z_0^+)>0$ and $\Delta \gamma(Z_1^+)>0$ such that $Z_0^+ \lor Z_1^+$ is not strictly contained in $X_M+e$.

    Since both $\Delta \gamma(Z_0^+)$ and $\Delta \gamma(Z_1^+)$ are non zero we have that $e \in Z_0^+$ and $e \in Z_1^+$ by Lemma~\ref{first lemma} and so we let $Z_0=Z_0^+-e$ and $Z_1=Z_1^+-e$. Also, since both $\Delta \gamma(Z_0^+)$ and $\Delta \gamma(Z_1^+)$ are positive and $\gamma_{M^+}(Z_0^+)$ and $\gamma_{M^+}(Z_1^+)$ are nonnegative (since $M^+$ is a strict gammoid) we see that $\gamma_M(Z_0)$ and $\gamma_M(Z_1)$ are positive and hence nonzero, meaning they are cyclic flats of $M$ by Lemma~\ref{f lemma.5}. Also by Lemma~\ref{f lemma.2} we have that $(Z_0^+ \lor Z_1^+) -e = Z_0 \lor Z_1$. If $X_M \subseteq (Z_0 \lor Z_1)$ then we set $X_M=X$ and we are done so we assume that $X_M \not \subseteq (Z_0 \lor Z_1)$. Note that since $(Z_0 \lor Z_1) \not \subset X_M$ there exists some element in $Z_0 \lor Z_1$ that is not in $X_M$.

    For ease of notation we set $Y=Z_0 \lor Z_1$. We have that $Y$ is a cyclic flat of $M$. Since both $Z_0$ and $Z_1$ are contained in $X_M$ we clearly have $Z_0 \cup Z_1 \subseteq X_M$ and therefore since $r_M(Y)=r_M(Z_0 \cup Z_1)$ we have that $r_M(Y)=r_M(Y \cap X_M)$. This gives us
    \[
    n_M(Y) - n_M(X_M \cap Y) = |Y| - r_M(Y) - |Y\cap X_M| + r_M(Y \cap X_M)
    \]
    which implies
    \begin{equation}\label{nullity equation 1}
    n_M(Y) - n_M(X_M \cap Y) = |Y - X_M|.
    \end{equation}

    Then since $r_M(Y) = r_M(Y \cap X_M)$ this means that $Y \in cl_M(X_M)$ and therefore $r_M(Y \cup X_M) = r_M(X_M)$ and so we have
    \[
    n_M(X_M \cup Y) - n_M(X_M) = |X_M \cup Y| - r_M(X_M \cup Y) - |X_M| + r_M(X_M)
    \]
    \begin{equation}\label{nullity equation 2}
        n_M(X_M \cup Y) = |Y-X_M| + n_M(X_M).
    \end{equation}

    We now examine $\gamma_M(Y) - \gamma_M(X_M \cap Y)$.
    \begin{multline*}
        \gamma_M(Y) - \gamma_M(X_M \cap Y) = n_M(Y) - \left( \sum_{Z \in \mathcal{Z}(M): Z \subset Y} \gamma_M(Z) \right) \\ 
        - n_M(X_M \cap Y) + \sum_{Z \in \mathcal{Z}(M): Z \subset (X_M \cap Y)} \gamma_M(Z).
    \end{multline*}

    We then split the first sum into those terms that are properly contained in $X_M \cap Y$ and those that are not.

    \begin{multline*}
        \gamma_M(Y) - \gamma_M(X_M \cap Y) = n_M(Y) - \sum_{\substack{Z \in \mathcal{Z}(M): Z \subset Y\\ Z \not \subset (X_M \cap Y)}} \gamma_M(Z) \\ - \left( \sum_{Z \in \mathcal{Z}(M): Z \subset (X_M \cap Y)} \gamma_M(Z) \right) -
        n_M(X_M \cap Y) + \sum_{Z \in \mathcal{Z}(M): Z \subset (X_M \cap Y)} \gamma_M(Z).
    \end{multline*}

    This allows us to cancel and obtain
    \[
    \gamma_M(Y) - \gamma_M(X_M \cap Y) = n_M(Y) - \left( \sum_{\substack{Z \in \mathcal{Z}(M): Z \subset Y\\ Z \not \subset (X_M \cap Y)}} \gamma_M(Z) \right) - n_M(X_M \cap Y).
    \]
    
    This rearranges to
    \[
    \sum_{\substack{Z \in \mathcal{Z}(M): Z \subseteq Y\\ Z \not \subset (X_M \cap Y)}} \gamma_M(Z) = \gamma_M(X_M \cap Y) + n_M(Y) - n_M(X_M \cap Y)
    \]
    where we have added the $\gamma_M(Y)$ term to the sum by increasing the range from $Z \subset Y$ to $Z \subseteq Y$. We now use \eqref{nullity equation 1} to obtain
    \[
    \sum_{\substack{Z \in \mathcal{Z}(M): Z \subseteq Y\\ Z \not \subset (X_M \cap Y)}} \gamma_M(Z) = \gamma_M(X_M \cap Y) + |Y-X_M|.
    \]

    If $\gamma_M(X_M \cap Y)<0$ then since $Z_0$ and $Z_1$ are both contained in $X_M \cap Y$ and $Y$ contains $X_M \cap Y$ we may set $X=X_M \cap Y$ and be done. Hence we may assume that $\gamma_M(X_M \cap Y)\geq 0$ and so we have
    \begin{equation}\label{abs value}
        \sum_{\substack{Z \in \mathcal{Z}(M): Z \subseteq Y\\ Z \not \subset (X_M \cap Y)}} \gamma_M(Z) \geq |Y-X_M|.
    \end{equation}

    We now examine $\gamma_M(X_M \cup Y)$.
    \[
        \gamma_M(X_M \cup Y) = n_M(X_M \cup Y) - \sum_{Z \in \mathcal{Z}(M): Z \subset (X_M \cup Y)} \gamma_M(Z).
    \]

    We split up the sum into terms that are strictly contained in $X_M$, the terms that are contained in $Y$ but not in $X_M$ and all the other terms, to get
    \begin{multline*}
        \gamma_M(X_M \cup Y) = n_M(X_M \cup Y) - \sum_{Z \in \mathcal{Z}(M): Z \subset X_M} \gamma_M(Z)\\
        -\sum_{\substack{Z \in \mathcal{Z}(M): Z \subseteq Y\\ Z \not \subset X_M}} \gamma_M(Z) - \sum_{\substack{Z \in \mathcal{Z}(M): Z \subset (X_M \cup Y)\\ Z \not \subset X_M: Z \not \subseteq Y}} \gamma_M(Z)
    \end{multline*}
    where we are including $Y$ in the second sum since $Y$ is a cyclic flat of $M$. We now substitute in \eqref{nullity equation 2} to get
    \begin{multline*}
        \gamma_M(X_M \cup Y) = |Y-X_M| + n_M(X_M)  - \sum_{Z \in \mathcal{Z}(M): Z \subset X_M} \gamma_M(Z) - \\
        \sum_{\substack{Z \in \mathcal{Z}(M): Z \subseteq Y\\ Z \not \subset X_M}} \gamma_M(Z) - \sum_{\substack{Z \in \mathcal{Z}(M): Z \subset (X_M \cup Y)\\ Z \not \subset X_M: Z \not \subseteq Y}} \gamma_M(Z).
    \end{multline*}

    But then the second and third terms on the right hand side simply add to make $\gamma_M(X_M)$ and so we have
    \[
    \gamma_M(X_M \cup Y) = \gamma_M(X_M) + |Y-X_M| - \sum_{\substack{Z \in \mathcal{Z}(M): Z \subseteq Y\\ Z \not \subset X_M}} \gamma_M(Z) - \sum_{\substack{Z \in \mathcal{Z}(M): Z \subset (X_M \cup Y)\\ Z \not \subset X_M: Z \not \subseteq Y}} \gamma_M(Z).
    \]
    
    However, since the range of the first sum is all the cyclic flats that are in $Y$ but not in $X_M$ we could equally write it as all the cyclic flats that are in $Y$ but not in $X_M \cap Y$ since this would exclude the same cyclic flats. This gives us
    \[
    \gamma_M(X_M \cup Y) = \gamma_M(X_M) + |Y-X_M| - \sum_{\substack{Z \in \mathcal{Z}(M): Z \subseteq Y\\ Z \not \subset (X_M \cap Y)}} \gamma_M(Z) - \sum_{\substack{Z \in \mathcal{Z}(M): Z \subset (X_M \cup Y)\\ Z \not \subset X_M: Z \not \subseteq Y}} \gamma_M(Z).
    \]

    We then use \eqref{abs value} to swap out the first sum for $|Y-X_M|$
    \[
    \gamma_M(X_M \cup Y) \leq \gamma_M(X_M) + |Y-X_M| - |Y-X_M| - \sum_{\substack{Z \in \mathcal{Z}(M): Z \subset (X_M \cup Y)\\ Z \not \subset X_M: Z \not \subseteq Y}} \gamma_M(Z)
    \]
    which cancels to give
    \[
    \gamma_M(X_M \cup Y) \leq \gamma_M(X_M) - \sum_{\substack{Z \in \mathcal{Z}(M): Z \subset (X_M \cup Y)\\ Z \not \subset X_M: Z \not \subseteq Y}} \gamma_M(Z)
    \]
    but every term on the right hand side is negative since all cyclic flats of $M$ have nonnegative $\gamma_M$ values and $\gamma_M(X_M)$ is negative by definition. Hence $\gamma_M(X_M \cup Y)<0$ but since $X_M \cup Y$ is the union of a cyclic set and a cyclic flat it is a cyclic set of $M$ with a negative $\gamma_M$ value, contradicting the maximality of $X_M$ and therefore we obtain the result.
\end{proof}

We can verify in polynomial time whether or not a set such as the one described in Lemma~\ref{final lemma} exists or not. To do this we examine every pair of cyclic flats in $M$ with positive $\gamma_M$ values and compute the rank of their union and the closure of their union in both $M$ and $M'$. If either of these disagree then clearly $M \neq M'$. Otherwise those two cyclic flats cannot be the two cyclic flats contained in $X$.

To see this let the two cyclic flats be $Z_0$ and $Z_1$. If $X$ is contained in $cl_M(Z_0 \cup Z_1)$ then since $(Z_0 \cup Z_1) \subseteq X$ we have that $r_M(X) = r_M(Z_0 \cup Z_1)$ and so if 
    \[
    r_{M'}(Z_0 \cup Z_1) = r_M(Z_0 \cup Z_1)
    \]
and
    \[
    cl_{M'}(Z_0 \cup Z_1) = cl_M(Z_0 \cup Z_1)
    \]
then 
    \[
    r_{M'}(X)= r_{M'}(cl_{M'}(Z_0 \cup Z_1)) = r_M(cl_M(Z_0 \cup Z_1))=r_M(X)
    \]
and so $n_M'(X)=n_M(X)$. But then by \eqref{F equation} we would have $\gamma_M(X)=\gamma_{M'}(X)$ but of course since $M'$ is a strict gammoid $\gamma_{M'}(X)\geq 0$ and so $X$ cannot be the set described by Lemma~\ref{final lemma}.

Fortunately there are only polynomially many such pairs to check as $Z_0$ and $Z_1$ are both in $\mathcal{F}$ and the size of $\mathcal{F}$ is bounded by $n_M(M)$ which is bounded by $n=E(M)$, so there are at most $n(n-1)/2$
pairs to check which can be done in polynomial time. Hence we have constructed an algorithm which will verify in polynomial time in the size of $M^+$ whether or not $M$ is a strict gammoid. Therefore we provide a proof of Theorem~\ref{theorem 2} and therefore a proof of Theorem~\ref{theorem 1}.
\section{Future Work}

Having put \textsc{Single Element Strict Gammoid Deletion} into $P$ we naturally examine the clear extension of the problem. Once again any strict gammoids appearing in problem instances will be assumed to be represented by directed graphs in the usual way.\\
\\
\textsc{Strict Gammoid Deletion}\\
\textsc{Instance}: A strict gammoid $M$ and a subset $X \subseteq E(M)$.\\
\textsc{Question}: Is $M\backslash X$ a strict gammoid?\\
\\
This problem is clearly the dual problem of \textsc{Transversal Contraction}. We may also examine the dual problem of \textsc{Transversal Dual}, which is\\
\\
\textsc{Strict Gammoid Dual}\\
\textsc{Instance}: A strict gammoid $M$.\\
\textsc{Question}: Is $M^*$ a strict gammoid?\\
\\
A polynomial-time algorithm to solve \textsc{Strict Gammoid Dual} would provide a solution to \cite[Problem 29]{DominicProblems} and \cite[Open Problem 6.1]{jobonotes} and so is of great interest. We examine how one might use similar methods to ours to obtain such a solution.

Let $M=M(D,V(D),S)$ be a strict gammoid. There exists a polynomial-time algorithm \cite{oxley} that obtains a representation of $M^*$ as a bipartite graph $G$. We can turn $G$ into a representation of a gammoid, $M_G$, by directing edges from $V$ into $A$ and letting $A$ be the target set \cite{oxley}. Now, sets of paths that link sets from $V$ into $A$ correspond to matchings in $G$. Everything we have done so far can be done in polynomial time, all that remains is to determine, hopefully in polynomial time, if $M_G$ is a strict gammoid.

If we add all vertices in $A$ to $E(M_G)$ we will certainly obtain a strict gammoid, $M_S$. Then all that remains is to determine whether $M_S\backslash A$ is a strict gammoid. That is, we have reduced our instance of \textsc{Strict Gammoid Dual} to an instance of \textsc{Strict Gammoid Deletion}. Furthermore this instance of \textsc{Strict Gammoid Deletion} has the property that $X$, the set we are deleting, is a basis of $M_S$. In fact $X$ is a {\em fundamental basis} of $M_S$. That is, a basis whose intersection with any cyclic flat spans the cyclic flat.

Therefore in order to obtain a polynomial time algorithm for \textsc{Strict Gammoid Dual} and therefore for \textsc{Transversal Dual} it suffices to find a polynomial time algorithm for \textsc{Strict Gammoid Deletion} in the case when $X$ is a fundamental basis. This problem shares certain similarities with the problem we have just solved as we outline below. For the proofs of these results see the author's thesis.

Let $M^B$ be a strict gammoid with fundamental basis $B$. Now let $M=M^B \backslash B$. We once more define $f:\mathcal{Z}(M^B) \rightarrow 2^{E(M)}$ as $f(Z) = Z-B$. We examine the new version of Lemma~\ref{f lemma}

\begin{lemma}
\label{new f lemma}
    The function $f$ has the following properties for all $Z_0, Z_1 \in \mathcal{Z}(M^B)$ and $Z \in \mathcal{Z}(M)$:
    \begin{enumerate}[label=\ref{new f lemma}.\arabic*.,ref=\ref{new f lemma}.\arabic*]
        \item \label{newf1} If $f(Z_0)=f(Z_1)$ then $Z_0 = Z_1$.
        \item \label{newf2} $cl_M(f(Z_0) \cup f(Z_1)) \subseteq f(cl_{M^B}(Z_0 \cup Z_1)$.
        \item \label{newf3} There exists some $Z' \in \mathcal{Z}(M^B)$ such that $f(Z')=Z$. 
        \item \label{newf4} Either $f(Z_0) \in \mathcal{Z}(M)$ or $\gamma_M(f(Z_0))=0$.
    \end{enumerate}
\end{lemma}

We also have a new version of Lemma~\ref{first lemma}

\begin{lemma}
\label{new first lemma}
	Let $X$ be a cyclic flat in $M^B$. Then
	\[\gamma_{M^B}(X)  - \gamma_M(f(X)) -r_M(f(X)) = \sum_{Z \in \mathcal{Z}(M^B): Z \subset X} (\gamma_M(f(Z)) - \gamma_{M^B}(Z)).\]
\end{lemma}

As before we can rewrite this as
\[
-\Delta \gamma(X) - r_M(f(X)) = \sum_{Z \in \down(X)-X} \Delta \gamma(Z).
\]

Unfortunately the next step, where we take elements out of $\down(X) -X$ and cancel things out does not appear to work in this case. What seems to happen is that terms regarding the ranks of cyclic flats and their intersections appear which cannot be cancelled out so easily. However, we are still hopeful that a solution to this problem may exist, or another algorithm may exist.
\section{Conclusion}

In this paper we have presented two important open problems in \textsc{Transversal Contraction} and \textsc{Transversal Dual}. We have then obtained a polynomial-time algorithm that solves \textsc{Transversal Contraction} in the case of contracting a single element.

It is difficult to be sure whether or not these techniques can be extended to the multiple element case. Unfortunately they do not extend directly and if the wrong elements are deleted even Lemma~\ref{f lemma} ceases to hold which means that different cyclic flats of the original matroid can map to the same cyclic flat in the new matroid. This would likely be a problem for any technique similar to ours as then having the sums of $\Delta \gamma$ terms line up would be more difficult.

In the case of simply deleting a fundamental basis things seem slightly more hopeful, however, as in this case Lemma~\ref{f lemma} mostly holds in the form of Lemma~\ref{new f lemma} and Lemma~\ref{first lemma} similarly holds in a slightly different form, that of Lemma~\ref{new first lemma}. While the other results do not hold we are nevertheless optimistic that either different versions of these results exist or that the problem will prove to be tractable using some other technique. Therefore we make the following conjecture:

\begin{conjecture}
    There exists a polynomial time algorithm that solves \textsc{Strict Gammoid Deletion} in the case when $X$ is a fundamental basis. If this algorithm obtains a positive answer then it also provides a strict gammoid representation of $M\backslash X$.
\end{conjecture}

As explained in Section 5 this means that we are also making the following conjecture:

\begin{conjecture}
    There exists a polynomial time algorithm that solves \textsc{Transversal Dual}. If this algorithm obtains a positive answer then it also provides a transversal representation of $M^*$.
\end{conjecture}




\bibliographystyle{plain}
\bibliography{Matroidbib}

\end{document}